\documentclass[12pt,a4paper]{amsart}
\usepackage{amsmath}
\usepackage[english]{babel}
\usepackage{verbatim}
\usepackage[shortlabels]{enumitem}
\usepackage{color}
\usepackage{tikz-cd}

\newtheoremstyle{mio}%
{}{} 
{\itshape}{} 
{\bfseries}{.}{ } 
{#1 #2\thmnote{~\mdseries(#3)}} 
\theoremstyle{mio}
\newtheorem{teor}{Theorem}[section]
\newtheorem{cor}[teor]{Corollary}
\newtheorem{prop}[teor]{Proposition}
\newtheorem{lemma}[teor]{Lemma}
\newtheorem{defin}[teor]{Definition}

\newtheoremstyle{definition2}%
{}{} 
{}{} 
{\bfseries}{.}{ } 
{#1 #2\thmnote{\mdseries~ #3}} 
\theoremstyle{definition2}
\newtheorem{ex}[teor]{Example}

\title[Integer-valued polynomials and the Picard group]{Localizations of integer-valued polynomials and of their Picard group}
\author{Dario Spirito}
\date{\today}
\address{Dipartimento di Scienze Matematiche, Informatiche e Fisiche, Universit\`a degli Studi di Udine, Udine, Italy}
\email{dario.spirito@uniud.it}

\newcommand{\Max}{\mathrm{Max}}
\newcommand{\Over}{\mathrm{Over}}
\newcommand{\Spec}{\mathrm{Spec}}
\newcommand{\Zar}{\mathrm{Zar}}
\newcommand{\Int}{\mathrm{Int}}
\newcommand{\Inv}{\mathrm{Inv}}
\newcommand{\Pic}{\mathrm{Pic}}
\DeclareMathOperator{\coker}{coker}
\newcommand{\picpol}{\mathcal{P}}
\newcommand{\njaff}{\mathcal{N}}
\newcommand{\insN}{\mathbb{N}}
\newcommand{\insZ}{\mathbb{Z}}
\newcommand{\insQ}{\mathbb{Q}}
\newcommand{\inN}{\in\insN}

\begin{document}

\begin{abstract}
We prove a necessary and sufficient criterion for the ring of integer-valued polynomials to behave well under localization. Then, we study how the Picard group of $\Int(D)$ and the quotient group $\picpol(D):=\Pic(\Int(D))/\Pic(D)$ behave in relation to Jaffard, weak Jaffard and pre-Jaffard families; in particular, we show that $\picpol(D)\simeq\bigoplus\picpol(T)$ when $T$ ranges in a Jaffard family of $D$, and study when similar isomorphisms hold when $T$ ranges in a pre-Jaffard family. In particular, we show that the previous isomorphism holds when $D$ is an almost Dedekind domain such that the ring integer-valued polynomials behave well under localization and such that the maximal space of $D$ is scattered with respect to the inverse topology.
\end{abstract}

\maketitle

\section{Introduction}
Let $D$ be an integral domain with quotient field $K$. A polynomial $f(X)\in K[X]$ is \emph{integer-valued} over $D$ if $f(d)\in D$ for every $d\in D$; the set of all integer-valued polynomials is a ring, denoted by $\Int(D)$. The ring of integer-valued polynomials presents several properties that makes it a very interesting subject of study: for example, it is a simple example of a construction that does not involve limits, infinite families of indeterminates, or intersections of complicated families of rings, and that rather consistently produces rings that are non-Noetherian, even starting from a Noetherian ring. Furthermore, this construction can be tailored to several topics (for example, considering polynomials that are integer-valued only on a subset) in order to obtain examples of phenomenon that are difficult to obtain with other constructions. We refer the reader to the book \cite{intD} for background and results about integer-valued polynomials.

One particular problem of the theory of integer-valued polynomials is its relationship with localization: given a domain $D$ and a multiplicatively closed set, under what hypothesis the equality $S^{-1}\Int(D)=\Int(S^{-1}D)$ holds? Several special cases have been proved (see e.g. \cite[Section 1.2]{intD}, \cite[Proposition 2.1]{mori-Intvalued}, \cite{elliott-Int}); we give in Section \ref{sect:loc} a necessary and sufficient criterion for this to happen, involving the conductor $(D:f(D))$, and show how the known criterion descend from ours. We also deal not only with localizations but, more generally, with flat overrings of the base domain $D$.

We then concentrate on generalizing globalization properties for the Picard group $\Pic(\Int(D))$ of $\Int(D)$. Unless $\Int(D)$ is trivial (i.e., unless $\Int(D)=D[X]$), the Picard group of $\Int(D)$ is usually much larger than the Picard goup of $D$, and can be calculated only in very special circumstances (for example, for discrete valuation domains and for some kinds of one-dimensional Noetherian local domains \cite[Chapter 6]{intD}). To obtain a description of $\Pic(\Int(D))$ in more cases, the main tool is globalization: for example, when $D$ is a one-dimensional Noetherian domain, there is always an exact sequence
\begin{equation*}
0\longrightarrow\Pic(D)\longrightarrow\Pic(\Int(D))\longrightarrow\bigoplus_{M\in\Max(D)}\Pic(\Int(D_M))\longrightarrow 0,
\end{equation*}
which allows at least to understand the main features of $\Pic(\Int(D))$. In this context, our first result (given in two different forms in Theorems \ref{teor:ext-pic-jaffard} and \ref{teor:ext-jaff-picpol}) gives a generalization of the previous exact sequence, proving that a similar result holds if, instead of the family $\{D_M\mid M\in\Max(D)\}$, one takes a \emph{Jaffard family} of $D$, a particular kind of family of flat overrings with strong independence properties (see Section \ref{sect:prelim:jaff} for a precise definition). The result becomes more striking using the \emph{int-polynomial Picard group} $\picpol(D)$, defined as the quotient between $\Pic(\Int(D))$ and the image of the canonical inclusion of $\Pic(D)$: in this terminology, the theorem guarantees that $\picpol(D)$ and the direct sum $\bigoplus\{\picpol(T)\mid T\in\Theta\}$ are isomorphic for any Jaffard family $\Theta$.

In Sections \ref{sect:weakJaff} and \ref{sect:preJaff}, we further generalize this result by considering \emph{weak Jaffard families} and \emph{pre-Jaffard families}, that are obtained by relaxing the conditions defining a Jaffard family. In the former case, we obtain in Theorem \ref{teor:weakjaff-picpol} an exact sequence
\begin{equation*}
0\longrightarrow\bigoplus_{\substack{T\in\Theta\\ T\neq T_\infty}}\picpol(T)\longrightarrow\picpol(D)\longrightarrow\picpol(D,T_\infty)\longrightarrow 0
\end{equation*}
(see below for the definition of $T_\infty$ and $\picpol(D,T_\infty)$). For pre-Jaffard families, we use the result on weak Jaffard families to set up a transfinite inductive reasoning (which uses the \emph{derived sequence} of the pre-Jaffard family, see Section \ref{sect:prelim:jaff}) that allows to prove, under some additional hypothesis, the existence of an exact sequence
\begin{equation*}
0\longrightarrow\bigoplus_{T\in\Theta\setminus T_\alpha}\picpol(T)\longrightarrow\picpol(D)\longrightarrow\picpol(D,T_\alpha)\longrightarrow 0
\end{equation*}
(see below for the definition of $T_\alpha$). In particular, when also the pre-Jaffard family is \emph{sharp}, one obtain an isomorphism $\picpol(D)\simeq\bigoplus\{\picpol(T)\mid T\in\Theta\}$, just like in the case of Jaffard families. In particular, such an isomorphism holds when $D$ is an almost Dedekind domain such that $\Int(D)$ behaves well under localization and such that the maximal space of $D$ is scattered in the inverse topology (Corollary \ref{cor:almded-best}).

\section{Preliminaries}
Throughout the paper, $D$ is an integral domain with quotient field $K$.

An \emph{overring} of $D$ is a ring contained between $D$ and $K$; we denote by $\Over(D)$ the set of all overrings of $D$. A \emph{flat overring} is an overring that is flat as a $D$-module; in particular, every localization and every quotient ring of $D$ is a flat overring. If $T$ is a flat overring of $D$, then for every prime ideal $P$ of $T$ we have $T_P=D_{P\cap D}$; in particular, every flat overring is an intersection of localizations of $D$, and every (prime) ideal of $T$ is the extension of a (prime) ideal of $D$ \cite{akiba_gqr,richamn_generalized-qr}.

Let $I$ be a $D$-submodule of $K$ and $A\subseteq K$. The \emph{conductor} of $A$ in $I$ is $(I:A):=\{x\in K\mid xA\subseteq I\}$; moreover, $(I:A)=(I:AD)$, where we denote by $AD$ the $D$-submodule generated by $A$. The conductor is always a $D$-submodule of $K$, and can be $(0)$. If $T$ is a flat overring of $D$ and $J$ is a finitely generated $D$-module, then $(I:J)T=(IT:JT)$ \cite[Theorem 7.4]{matsumura}.

A \emph{fractional ideal} of $D$ is a $D$-submodule $I$ of $K$ such that $(D:I)\neq(0)$, i.e. such that $xI\subseteq D$ for some nonzero $x\in K$. A fractional ideal $I$ is \emph{invertible} if there is a fractional ideal $J$ such that $IJ=D$; equivalently, $I$ is invertible if it is finitely generated and locally principal (i.e., $ID_M$ is principal for every $M\in\Max(D)$). The set of invertible ideal is an abelian group, denoted by $\Inv(D)$, having as a subgroup the set $\mathrm{Princ}(D)$ of principal fractional ideals of $D$; the quotient $\Inv(D)/\mathrm{Princ}(D)$ is called the \emph{Picard group} of $D$, and is denoted by $\Pic(D)$.

\subsection{Topologies}
Let $D$ be an integral domain. The spectrum $\Spec(D)$ of $D$ can be endowed, in addition the usual Zariski topology, with another topology, called the \emph{inverse topology}. The inverse topology is defined as the topology having, as a subbasis of open sets, the closed sets of the Zariski topology. Under the inverse topology, the spectrum is still a compact $T_1$ space.

The set $\Over(D)$ of the overrings of $D$ can be endowed with a natural topology, called the \emph{Zariski topology}, whose subbasic open sets are the ones in the form
\begin{equation*}
B(x_1,\ldots,x_n):=\{T\in\Over(D)\mid x_1,\ldots,x_n\in T\},
\end{equation*}
as $x_1,\ldots,x_n$ varies in $K$. The Zariski topology on $\Over(D)$ is intimately connected with the Zariski topology on the spectrum $\Spec(D)$ of $D$: for example, the localization map $P\mapsto D_P$ is a topological embedding when $\Spec(D)$ and $\Over(D)$ are endowed with the respective Zariski topologies \cite[Lemma 2.4]{dobbs_fedder_fontana}. Moreover, the Zariski topology has several good properties: for example, it is a \emph{spectral space}, in the sense that there is a ring $R$ (not explicitly constructed) such that $\Spec(R)\simeq\Over(D)$ \cite[Proposition 3.5]{finocchiaro-ultrafiltri}.

The \emph{inverse topology} on $\Over(D)$ is the topology such that the $B(x_1,\ldots,x_n)$ are a subbasis of closed sets. This topology is closely connected with the properties of representations of $D$ as intersection of overrings (see e.g. \cite{olberding_topasp}). Properties of the inverse topology, in the context of spectral spaces, can be found in \cite{spectralspaces-libro}.

We shall use many times the following result \cite[Corollary 5]{compact-intersections}: if $\Theta\subseteq\Over(D)$ is compact, with respect to the Zariski topology, and if $I$ is a flat $D$-submodule of $K$, then
\begin{equation*}
I\left(\bigcap_{T\in\Theta}T\right)=\bigcap_{T\in\Theta}IT.
\end{equation*}

\subsection{Jaffard and pre-Jaffard families}\label{sect:prelim:jaff}
$D$ be an integral domain with quotient field $K$. We say that a subset $\Theta\subseteq\Over(D)$ is a \emph{pre-Jaffard family} of $D$ if the following conditions hold \cite{jaff-derived}:
\begin{itemize}
\item either $\Theta=\{K\}$ or $K\notin\Theta$;
\item every $T\in\Theta$ is flat over $D$;
\item $\Theta$ is \emph{complete}: if $I$ is an ideal of $D$, then $I=\bigcap\{IT\mid T\in\Theta\}$;
\item $\Theta$ is \emph{independent}: if $T\neq T'$ are in $\Theta$, then $TT'=K$;\footnote{This is not the best way to define independence for general families of overrings, but it is equivalent for flat overrings, and is the property we will be using; see \cite[Lemma 3.4 and Definition 3.5]{jaff-derived}.}
\item $\Theta$ is compact in the Zariski topology.
\end{itemize}
For example, if $D$ is a one-dimensional domain, the family $\Theta=\{D_M\mid M\in\Max(D)\}$ is a pre-Jaffard family of $D$.

In particular, if $\Theta$ is a pre-Jaffard family and $P$ is a nonzero prime ideal of $D$, then there is exactly one $T\in\Theta$ such that $PT\neq T$.

A family $\Theta$ of overrings of $D$ is \emph{locally finite} if every nonzero $x\in D$ is a nonunit in only finitely many elements of $\Theta$; if $\Theta=\{D_M\mid M\in\Max(D)\}$ is locally finite, we say that $D$ itself is locally finite. Any locally finite family of overrings is compact, with respect to the Zariski topology (see e.g. the proof of \cite[Corollary 8]{compact-intersections}).

A \emph{Jaffard family} is a pre-Jaffard family that is locally finite. Jaffard families enjoy several good factorization properties that make them a non-local generalization of $h$-local domains and thus of Dedekind domains; see for example \cite[Section 6.3]{fontana_factoring}, \cite[Section 4]{starloc} or \cite{length-funct}.

We say that an overring $T$ of $D$ is a \emph{Jaffard overring} if $T$ belongs to a Jaffard family of $D$. Given a Jaffard family $\Theta$ of $D$, we can construct a well-ordered decreasing chain $\{\njaff^\alpha(D)\}$ of subsets of $\Theta$ and a corresponding ascending chain $\{T_\alpha\}$ of overrings of $D$ in the following way. Given an ordinal number $\alpha$, we set: \cite[Section 6]{jaff-derived}
\begin{itemize}
\item if $\alpha=0$, $\njaff^0(\Theta):=\Theta$ and $T_0:=D$;
\item $\displaystyle{T_\alpha:=\bigcap_{T\in\njaff^\alpha(D)}T}$.
\item if $\alpha=\gamma+1$ is a limit ordinal, then $\njaff^\alpha(D)$ is the set of all elements of $\njaff^\gamma(D)$ that are \emph{not} Jaffard overrings of $T_\gamma$;
\item if $\alpha$ is a limit ordinal, then $\displaystyle{\njaff^\alpha(D):=\bigcap_{\beta<\alpha}\njaff^\beta(D)}$.
\end{itemize}
Note that, in \cite{jaff-derived}, the set $\njaff^\alpha(D)$ was denoted simply by $\Theta_\alpha$. Each $\njaff^\alpha(D)$ is a pre-Jaffard family of $T_\alpha$ (in particular, it is compact with respect to the Zariski topology) \cite[Proposition 6.1]{jaff-derived}, and it is a closed subset of $\Theta$, with respect to the inverse topology. We call $\{T_\alpha\}$ the \emph{derived sequence} of $\Theta$.

If $\njaff^1(\Theta)$ is a single element $T_\infty$, we say that $\Theta$ is a \emph{weak Jaffard family pointed at $T_\infty$}. Weak Jaffard families are usually the stepping stones in inductive arguments used to generalize properties of Jaffard families to pre-Jaffard families.

When $\njaff^\alpha(\Theta)=\emptyset$ for some $\alpha$ (equivalently, when $T_\alpha=K$) we say that $\Theta$ is \emph{sharp}.

\subsection{Homology}
We shall frequently use a basic results of homological algebra, the \emph{snake lemma}: if
\begin{equation*}
\begin{tikzcd}
0\arrow[r] & A_1\arrow[r]\arrow[d,"f"] & B_1\arrow[r]\arrow[d,"g"]  & C_1\arrow[r]\arrow[d,"h"] & 0\\
0\arrow[r] & A_2\arrow[r] & B_2\arrow[r]  & C_2\arrow[r] & 0\\
\end{tikzcd}
\end{equation*}
is a commutative diagram of abelian groups (or, more generally, of modules over a ring $R$) with exact rows, then the sequence
\begin{equation*}
0\longrightarrow\ker(f)\longrightarrow\ker(g)\longrightarrow\ker(h)\longrightarrow\coker(f)\longrightarrow\coker(g)\longrightarrow\coker(h)\longrightarrow 0
\end{equation*}
is exact. In particular, if $f,g,h$ are injective, then the sequence of cokernels
\begin{equation*}
0\longrightarrow\coker(f)\longrightarrow\coker(g)\longrightarrow\coker(h)\longrightarrow 0
\end{equation*}
is exact.

\section{When integer-valued polynomials localize}\label{sect:loc}

In this section, we find a necessary and sufficient criterion for $\Int(D)$ to localize at a flat overring, i.e., for when the equality $\Int(D)T=\Int(T)$ holds. Before doing so, we introduce a notion that generalizes Jaffard families and Jaffard overrings.

\begin{defin}
Let $D$ be an integral domain and $\Theta\subseteq\Over(D)$. We say that $\Theta$ is a \emph{$t$-Jaffard family} of $D$ if:
\begin{itemize}
\item either $\Theta=\{K\}$ or $K\notin\Theta$;
\item every $T\in\Theta$ is flat over $D$;
\item $\Theta$ is independent;
\item $\Theta$ is locally finite;
\item $\bigcap\{T\mid T\in\Theta\}=D$.
\end{itemize}
We say that an overring $T$ of $D$ is a \emph{$t$-Jaffard overring} if it belongs to a $t$-Jaffard family of $D$.
\end{defin}

Note that, in particular, every Jaffard family is a $t$-Jaffard family, and thus every Jaffard overring is a $t$-Jaffard overring. The converse does not hold: for example, if $D$ is a Krull domain, then the family of localizations at its prime ideals of height $1$ is a $t$-Jaffard family, but it is not a Jaffard family unless $D$ has dimension $1$.

The following proposition can be seen as a variant of \cite[Lemma 5.3]{starloc}.
\begin{prop}\label{prop:colon-tjaff}
Let $T$ be a $t$-Jaffard overring of $D$. Then, for every fractional ideal $I$ of $D$, we have $(D:I)T=(T:IT)$.
\end{prop}
\begin{proof}
Let $\Theta$ be a $t$-Jaffard family of $D$ containing $T$, and let $A:=\bigcap\{S\mid S\in\Theta\setminus\{T\}\}$. Since $\Theta\setminus\{T\}$ is locally finite, it is compact, and thus, by \cite[Corollary 5]{compact-intersections}
\begin{equation*}
AT=\left(\bigcap_{S\in\Theta\setminus\{T\}}S\right)T=\bigcap_{S\in\Theta\setminus\{T\}}ST=K.
\end{equation*}
Hence,
\begin{equation*}
(D:I)T=(T\cap A:I)T=((T:I)\cap(A:I))T=(T:I)T\cap(A:I)T.
\end{equation*}
We have $(T:I)T=(T:IT)T=(T:IT)$; on the other hand, $(A:I)$ is an $A$-ideal, and thus $(A:I)T=(A:I)AT=K$. Hence, $(D:I)T=(T:IT)$, as claimed.
\end{proof}

We now go back to studying integer-valued polynomials. The following is a slight generalization of \cite[Theorem I.2.1]{intD}.
\begin{prop}
Let $T$ be a flat overring of $D$ and $f\in K[X]$. Then $f(D)T=f(T)T$.
\end{prop}
\begin{proof}
Clearly $f(D)\subseteq f(T)$, and thus $f(D)T\subseteq f(T)T$. Conversely, let $t\in T$: we need to show that $f(t)\in f(D)T$. Consider $I:=(f(D)T:f(t))=(f(D)D:f(t))T$. If $f(t)\notin f(D)T$, then $(f(D)D:f(t))\subseteq P$ for some prime ideal $P$ such that $PT\neq T$; hence, $1\notin(f(D)D:f(t))D_P=(f(D)D_P:f(t))$. However, $f(D)D_P=f(D_P)D_P$ by \cite[Theorem I.2.1]{intD} (since $D_P$ is a localization of $D$), and $f(t)\in D_P$ since $T\subseteq D_P$ (by the flatness of $T$). This is a contradiction, and thus $f(t)\in f(D)T$ and $f(D)T=f(T)T$.
\end{proof}

\begin{prop}\label{prop:f-in-loc}
Let $T$ be a flat overring of $D$, and let $f\in K[X]$.
\begin{enumerate}[(a)]
\item\label{prop:f-in-loc:DT} $f\in\Int(D)T$ if and only if $(D:_Df(D))T=T$.
\item\label{prop:f-in-loc:T} $f\in\Int(T)$ if and only if $(T:_Tf(D)T)=T$.
\end{enumerate}
\end{prop}
\begin{proof}
\ref{prop:f-in-loc:DT} If $(D:_Df(D))T=T$, then $1=d_1t_1+\cdots+d_nt_n$ for some $d_i\in(D:_Df(D))$, $t_i\in T$; hence
\begin{equation*}
f(X)=f(X)(d_1t_1+\cdots+d_nt_n)=\sum_{i=1}^n(f(X)d_it_i).
\end{equation*}
However, $f(d)d_i\in D$ for all $d\in D$, since $d_i\in(D:_Df(D))$, and thus each $f(X)d_i\in\Int(D)$. Hence $f(X)\in\Int(D)T$. Conversely, suppose $f\in\Int(D)T$. If $T=D_P$ for some prime ideal $P$, then $\Int(D)D_P=\Int(D)_P$ and thus there is an $s\in D\setminus P$ such that $sf\in\Int(D)$; hence $s\in(D:_Df(D))$ and $(D:_Df(D))D_P=D_P$. For the general case, if $f\in\Int(D)T$ then $f\in\Int(D)D_P$ for all prime ideals $P$ of $D$ such that $PT\neq T$ (as $T$ is flat, $T\subseteq D_P$ for all such $P$), and thus $(D:_Df(D))D_P=D_P$; the claim now follows from the fact that all maximal ideals of $T$ are extensions of prime ideals of $D$.

\ref{prop:f-in-loc:T} If $f\in\Int(T)$ then $f(T)\subseteq T$ and thus $f(D)T\subseteq T$; hence $(T:_Tf(D)T)=T$. Conversely, if $(T:_Tf(D)T)=T$ then it contains $1$, and thus $f(D)T\subseteq T$. Since $f(D)T=f(T)T$, we have $f(T)\subseteq T$ and $f\in\Int(T)$.
\end{proof}

Joining the two characterizations, we have our criterion.
\begin{teor}\label{teor:loc}
Let $T$ be a flat overring of $D$. Then, $\Int(D)T=\Int(T)$ if and only if $(D:_Df(D))T=(T:_Tf(D)T)$ for every $f\in K[X]$.
\end{teor}
\begin{proof}
It is enough to apply the two conditions of Proposition \ref{prop:f-in-loc}.
\end{proof}

As a consequence, we get back several known results about the possibility of localizing the ring $\Int(D)$. Recall that a \emph{Mori domain} is a domain whose divisorial ideals satisfy the ascending chain condition.
\begin{prop}\label{prop:localizz}
Let $T$ be a flat overring of $D$. Suppose that one of the following conditions hold:
\begin{enumerate}[(a)]
\item\label{prop:localizz:jaff} $T$ is a Jaffard overring of $D$;
\item\label{prop:localizz:tjaff} $T$ is a $t$-Jaffard overring of $D$;
\item\label{prop:localizz:1dim} $D$ is one-dimensional and locally finite;
\item\label{prop:localizz:noeth}\cite[Theorem I.2.3]{intD} $D$ is Noetherian;
\item\label{prop:localizz:mori}\cite[Proposition 2.1]{mori-Intvalued} $D$ is Mori.
\end{enumerate}
Then, we have $\Int(D)T=\Int(T)$.
\end{prop}
\begin{proof}
If $T$ is a Jaffard or $t$-Jaffard overring, then $(D:I)T=(T:IT)$ for all ideals $I$, and thus in particular for $I=f(D)D$. If $D$ is one-dimensional and locally finite, then every flat overring is a Jaffard overring and we are in the previous case.

If $D$ is Noetherian, then $f(D)D$ is finitely generated, and thus we can bring the flat overring inside the conductor.

If $D$ is Mori, then every ideal is strictly $v$-finite, and thus there is a finitely generated ideal $J\subseteq f(D)D$ such that $(D:f(D))=(D:J)$. Hence,
\begin{equation*}
(D:f(D))T=(D:J)T=(T:JT)\supseteq(T:f(D)T)\supseteq(D:f(D))T
\end{equation*}
and thus $(D:f(D))T=(T:f(D)T)$.
\end{proof}

We end this section with a result which will be useful alter.
\begin{prop}\label{prop:localizz-intersec}
Let $D$ be an integral domain, $T$ a flat overring and $\Lambda$ be a complete family of flat overrings of $T$. If $\Int(D)S=\Int(S)$ for every $S\in\Lambda$, then $\Int(D)T=\Int(T)$.
\end{prop}
\begin{proof}
Let $f\in\Int(T)$: then, $f(T)\subseteq T\subseteq S$ for every $S\in\Lambda$, i.e., $f\in\Int(T,S)=\Int(S)=\Int(D)S$. Hence,
\begin{equation*}
f\in\bigcap_{S\in\Lambda}\Int(D)S=\bigcap_{S\in\Lambda}\Int(D)TS=\Int(D)T,
\end{equation*}
since $\Lambda$ is complete over $T$. The claim is proved.
\end{proof}

\section{The Picard group}
When $\Int(D)$ is nontrival, a direct calculation of its Picard group can only be done under very special circumstances, for example when $D$ is a discrete valuation ring or an analytically irreducible one-dimensional domain \cite[Proposition VIII.2.8 and Corollary VIII.3.10]{intD}. To reach more cases, the main tool is globalization: for example, if $D$ is one-dimensional Noetherian domain, then there is a short exact sequence \cite[Theorem VIII.1.9]{intD}
\begin{equation*}
0\longrightarrow\Pic(D)\longrightarrow\Pic(\Int(D))\longrightarrow\bigoplus_{M\in\Max(D)}\Pic(\Int(D_M))\longrightarrow 0.
\end{equation*}
In this section, we begin to extend the use of this kind of exact sequence by considering the case of Jaffard families.

\begin{defin}
Let $T$ be a flat overring of $D$. Then \emph{extension map} of Picard groups is the group homomorphism
\begin{equation*}
\begin{aligned}
\psi_{D,T}\colon\Pic(D) & \longrightarrow\Pic(T),\\
[I] & \longmapsto [IT].
\end{aligned}
\end{equation*}

If $\Theta$ is a family of flat overrings of $D$, the \emph{Picard group of $D$ relative to $\Theta$} is
\begin{equation*}
\Pic(D,\Theta):=\{[I]\in\Pic(D)\mid [IT]=[T]\text{~for all~}T\in\Theta\}.
\end{equation*}
If $\Theta=\{T\}$ we write $\Pic(D,T):=\Pic(D,\{T\})$.
\end{defin}

Since $\psi_{D,T}$ is a group homomorphism, $\Pic(D,\Theta)$ is always a subgroup of $\Pic(D)$ (indeed, it is the intersection of the kernels of the $\psi_{D,T}$, as $T$ ranges in $\Theta$). When every element of $\Theta$ is local, $\Pic(D,\Theta)=\Pic(D)$.

The starting point of the globalization results of \cite[Chapter VIII]{intD} is an extension map from $\Int(D)$ to the direct product of $\Int(D)_M=\Int(D)D_M$, as $M$ ranges among the maximal ideals of $D$ \cite[Proposition VIII.1.6]{intD}. Likewise, our study begins by examining the extension map
\begin{equation}\label{eq:pitheta}
\begin{aligned}
\pi_\Theta\colon\Pic(\Int(D)) & \longrightarrow\prod_{T\in\Theta}\Pic(\Int(T)),\\
[I] & \longmapsto ([I\Int(T)]),
\end{aligned}
\end{equation}
for some arbitrary family $\Theta$ of flat overrings. 

\begin{lemma}
Let $T$ be a flat overring of $D$. Then, $\Int(D)T$ is a flat overring of $\Int(D)$.
\end{lemma}
\begin{proof}
Since $T$ is flat, it is the colimit of a directed set $\{M_i\}$ of free $D$-modules; since each of these is contained in the quotient field of $D$, there are $x_i$ such that $M_i=x_iD$. It is straightforward to see that $\Int(D)T$ is the colimit of $\{x_i\Int(D)=\Int(D)M_i\}$, and thus it is flat over $\Int(D)$.
\end{proof}

\begin{prop}\label{prop:prejaff-exact}
Let $\Theta$ be a family of flat overrings of $D$. Then, there is an exact sequence
\begin{equation*}
0\longrightarrow\Pic(D,\Theta)\longrightarrow\Pic(\Int(D))\xrightarrow{~~\pi_\Theta~~}\prod_{T\in\Theta}\Pic(\Int(D)T).
\end{equation*}
\end{prop}
\begin{proof}
Let $i:\Pic(D,\Theta)\longrightarrow\Pic(\Int(D))$, $I\mapsto I\Int(D)$ be the extension map; we need to show that $i(\Pic(D,\Theta))=\ker\pi_\Theta$. If $[I]\in\Pic(D,\Theta)$, then $I$ becomes principal in each $T\in\Theta$, and thus $\pi_\Theta\circ i([I])=[I\Int(D)T]$ is principal, i.e., $i(\Pic(D,\Theta))\subseteq\ker\pi_\Theta$.

Conversely, suppose $[I]\in\Pic(\Int(D))$ becomes principal in each $\Pic(\Int(D)T)$. By \cite[Remark VIII.1.5]{intD}, we can suppose without loss of generality that $I$ is a unitary ideal of $\Int(D)$. Let $J:=I\cap D$. For each $T$, the ideal $JT$ is principal and generated by an element of $J$; hence, $I\Int(D)T=J\Int(D)T$ for each $T\in\Theta$. Let $\Lambda:=\{\Int(D)T\mid T\in\Theta\}$; then, the map $\star_\Lambda:I\mapsto\bigcap\{IT\mid T\in\Lambda\}$ is a star operation (see e.g. \cite[Chapter 32]{gilmer}) and $I^{\star_\Lambda}=(J\Int(D))^{\star_\Lambda}$. Since $I$ is invertible, $I=I^{\star_\Lambda}$, and analogously $J\Int(D)=(J\Int(D))^{\star_\Lambda}$; thus $I=J\Int(D)$. Hence $[J]\in i(\Pic(D,\Theta))$ and $i(\Pic(D,\Theta))\supseteq\ker\pi_\Theta$, as claimed.
\end{proof}

We now prove the first theorem of this section.
\begin{teor}\label{teor:ext-pic-jaffard}
Let $\Theta$ be a Jaffard family of $D$. Then, there is an exact sequence
\begin{equation*}
0\longrightarrow\Pic(D,\Theta)\longrightarrow\Pic(\Int(D))\longrightarrow\bigoplus_{T\in\Theta}\Pic(\Int(T))\longrightarrow 0.
\end{equation*}
\end{teor}
\begin{proof}
We first note that, by Proposition \ref{prop:localizz}\ref{prop:localizz:jaff}, we have $\Int(D)T=\Int(T)$, and thus $\Pic(\Int(D)T)=\Pic(\Int(T))$.

Let $\Delta$ be the image of the extension map $\pi_\Theta:\Pic(\Int(D))\longrightarrow\prod\{\Pic(\Int(T))\mid T\in\Theta\}$. We claim that its image is just the direct sum.

Indeed, if $[I]\in\Pic(\Int(D))$ then by \cite[Remark VIII.1.5]{intD} we can suppose that $I$ is an integral unitary ideal of $\Int(D)$; in particular, $I$ contains a nonzero constant $a$. Since $\Theta$ is a Jaffard family, it is locally finite, and thus $aT=T$ for all but finitely many $T\in\Theta$; hence, $I\Int(T)=\Int(T)$ for all but finitely many $T$, and thus $[I\Int(T)]$ is almost always equal to $[\Int(T)]$. It follows that $\Delta$ is contained in the direct sum.

To prove the converse, it is enough to show that, for any $T\in\Theta$ and any $[J]\in\Pic(\Int(T))$, there is a $[I]\in\Pic(\Int(D))$ such that $[I\Int(T)]=[J]$ and $[I\Int(S)]=[\Int(S)]$ for every $S\in\Theta$, $S\neq T$. Again by \cite[Remark VIII.1.5]{intD} we can suppose that $J$ is an integral unitary ideal of $\Int(T)$; moreover, we can suppose that $J=(f_1,\ldots,f_n)\Int(T)$ for some $f_1,\ldots,f_n\in\Int(D)$. Let $L:=J\cap D=J\cap T\cap D=I\cap D$: then, $LS=S$ for every $S\in\Theta$, $S\neq T$. Then, $I':=(f_1,\ldots,f_n)\Int(D)+L\Int(D)$ is contained in $I$ and finitely generated, but $I'\Int(T)=J$ and $I'\Int(S)=\Int(S)$, as well as $I'K[X]=K[X]$. It follows that $I'=I$, and thus $I$ is finitely generated.

We show that $I$ is locally principal. Let $M$ be a maximal ideal of $\Int(D)$: if $M\cap D=(0)$ then $\Int(D)_M$ is a localization of $K[X]$, and thus $I\Int(D)_M$ is principal. If $M\cap D\neq(0)$, then $\Int(D)_M$ contains $\Int(D)S$ for some $S\in\Theta$, and thus $I\Int(D)_M=I\Int(S)\Int(D)_M$. If $S\neq T$, then $I\Int(D)_M=I\Int(S)\Int(D)_M=\Int(D)_M$ is principal. If $S=T$, then $I\Int(T)=J$ is invertible, and thus $I\Int(D)_M$ is principal since $\Int(D)_M$ is a localization of $\Int(T)$. Thus $I$ is locally principal and thus invertible; therefore the direct sum is in the image of $\psi_\Theta$. The claim is proved.
\end{proof}

While very similar to the localization result for Dedekind domains, Theorem \ref{teor:ext-pic-jaffard} includes in its statement the group $\Pic(D,\Theta)$, which may not be easy to calculate. In the next theorem, we trade its presence with the one of the Picard groups $\Pic(T)$; we first show how they are related.

\begin{prop}\label{prop:jaffard-PicDTheta}
Let $\Theta$ be a Jaffard family of $D$. Then, there is an exact sequence
\begin{equation*}
0\longrightarrow\Pic(D,\Theta)\longrightarrow\Pic(D)\longrightarrow\bigoplus_{T\in\Theta}\Pic(T)\longrightarrow 0.
\end{equation*}
\end{prop}
\begin{proof}
By \cite[Proposition 7.1]{starloc}, the extension map
\begin{equation*}
\begin{aligned}
\Gamma\colon\Inv(D) & \longrightarrow\bigoplus_{T\in\Theta}\Inv(T),\\
I & \longmapsto IT
\end{aligned}
\end{equation*}
is an isomorphism. Since every principal ideal of $D$ becomes principal in each $T$, $\Gamma$ induces a surjective map $\Gamma':\Pic(D)\longrightarrow\bigoplus\{\Pic(T)\mid T\in\Theta\}$, whose kernel by definition is exactly $\Pic(D,\Theta)$. The claim is proved.
\end{proof}

\begin{defin}
Let $D$ be an integral domain, and let $\iota_D:\Pic(D)\longrightarrow\Pic(\Int(D))$, $I\mapsto I\Int(D)$, be the canonical extension map. We define the \emph{int-polynomial Picard group of $D$} as the quotient
\begin{equation*}
\picpol(D):=\frac{\Pic(\Int(D))}{\iota_D(\Pic(D))}.
\end{equation*}
If $T$ is a flat overring of $D$, we also define the \emph{int-polynomial Picard group of $(D,T)$ as}
\begin{equation*}
\picpol(D,T):=\frac{\Pic(\Int(D)T)}{\iota_{D,T}(\Pic(T))},
\end{equation*}
where $\iota_{D,T}:\Pic(T)\longrightarrow\Pic(\Int(D)T)$ is the extension map.
\end{defin}

Note that, when $D$ is a local ring, $\Pic(D)=(0)$, and thus $\picpol(D)=\Pic(\Int(D))$.

\begin{teor}\label{teor:ext-jaff-picpol}
Let $\Theta$ be a Jaffard family of $D$. Then, there is an exact sequence
\begin{equation*}
0\longrightarrow\Pic(D)\longrightarrow\Pic(\Int(D))\longrightarrow\bigoplus_{T\in\Theta}\picpol(T)\longrightarrow 0.
\end{equation*}
In particular,
\begin{equation*}
\picpol(D)\simeq\bigoplus_{T\in\Theta}\picpol(T).
\end{equation*}
\end{teor}
\begin{proof}
Consider the commutative diagram
\begin{equation*}
\begin{tikzcd}
0\arrow[r] & \Pic(D,\Theta) \arrow[r]\arrow[d,equal] & \Pic(D) \arrow[r]\arrow[d,"\iota_D"] & \displaystyle{\bigoplus_{T\in\Theta}\Pic(T)}\arrow[d,"\iota_\Theta"] \arrow[r] & 0\\
0\arrow[r] & \Pic(D,\Theta) \arrow[r] & \Pic(\Int(D))\arrow[r] & \displaystyle{\bigoplus_{T\in\Theta}\Pic(\Int(T))}\arrow[r] & 0\\
\end{tikzcd}
\end{equation*}

The first row is exact by Proposition \ref{prop:jaffard-PicDTheta}, while the second one from Theorem \ref{teor:ext-pic-jaffard}; on the other hand, the leftmost vertical map is the identity and the other two vertical maps are injective. By the snake lemma, there is an exact sequence $0\longrightarrow\coker \iota_D\longrightarrow\coker \iota_\Theta\longrightarrow 0$. By definition, $\coker\iota_D$ is just $\picpol(D)$, while $\coker\iota_\Theta$ is the direct sum $\bigoplus\picpol(T)$, and thus we have the isomorphism. The sequence (which is exact by definition)
\begin{equation*}
0\longrightarrow\Pic(D)\longrightarrow\Pic(\Int(D))\longrightarrow\picpol(D)\longrightarrow 0
\end{equation*}
then becomes the one in the statement by substituting $\picpol(D)$ with the direct sum.
\end{proof}

A domain is \emph{$h$-local} if every nonzero ideal is contained in only finitely many maximal ideals and every nonzero prime ideal is contained in only one maximal ideal. The previous theorems immediately give the following.

\begin{cor}
Let $D$ be an integral domain such that one of the following conditions holds.
\begin{enumerate}[(a)]
\item $D$ is $h$-local;
\item $D$ is one-dimensional and locally finite;
\item $D$ is a one-dimensional Noetherian domain.
\end{enumerate}
Then, there is an exact sequence
\begin{equation*}
0\longrightarrow\Pic(D)\longrightarrow\Pic(\Int(D))\longrightarrow\bigoplus_{M\in\Max(D)}\Pic(\Int(D_M))\longrightarrow 0.
\end{equation*}
In particular, $\displaystyle{\picpol(D)\simeq\bigoplus_{M\in\Max(D)}\picpol(D_M)}\simeq\bigoplus_{M\in\Max(D)}\Pic(\Int(D_M))$.
\end{cor}
\begin{proof}
We first note that, if $D$ is one-dimensional and locally finite, then $D$ is $h$-local; likewise, if $D$ is one-dimensional and Noetherian, then it is locally finite. Hence it is enough to prove the claim for $D$ $h$-local.

If $D$ is $h$-local, $\Theta:=\{D_M\mid M\in\Max(D)\}$ is a Jaffard family, and thus the claim follows either from Theorem \ref{teor:ext-pic-jaffard} (since $\Pic(D,\Theta)=\Pic(D)$) or by Theorem \ref{teor:ext-jaff-picpol} (since $\picpol(D_M)=\Pic(\Int(D_M))$ as $D_M$ is local).
\end{proof}

\begin{prop}
Let $D$ be a locally finite Pr\"ufer domain. Then, there is a split exact sequence
\begin{equation*}
0\longrightarrow\Pic(D)\longrightarrow\Pic(\Int(D))\longrightarrow\bigoplus_{\substack{M\in\Max(D)\\ h(M)=1}}\Pic(\Int(D_M))\longrightarrow 0.
\end{equation*}
In particular,
\begin{equation*}
\Pic(\Int(D))\simeq\Pic(D)\oplus\bigoplus_{\substack{M\in\Max(D)\\ h(M)=1}}\Pic(\Int(D_M))
\end{equation*}
\end{prop}
\begin{proof}
Let $T:=\bigcap\{D_M\mid M\in\Max(D),h(M)>1\}$, and let $\Theta:=\{D_M\mid M\in\Max(D),h(M)=1\}\cup\{T\}$. Then, $\Theta$ is complete, locally finite, and each of its elements is flat over $D$. Moreover, $D_MD_N=K$ if $M,N$ have height $1$, while
\begin{equation*}
TD_N=\left(\bigcap_{\substack{M\in\Max(D)\\ h(M)>1}}D_M\right)D_N=\bigcap_{\substack{M\in\Max(D)\\ h(M)>1}}D_MD_N=K
\end{equation*}
since each subset of $\Max(D)$ is compact. Hence, $\Theta$ is independent and thus a Jaffard family. By Theorem \ref{teor:ext-pic-jaffard} there is an exact sequence
\begin{equation*}
0\longrightarrow\Pic(D)\longrightarrow\Pic(\Int(D))\longrightarrow\picpol(T)\oplus\bigoplus_{\substack{M\in\Max(D)\\ h(M)=1}}\picpol(D_M)\longrightarrow 0.
\end{equation*}
Each $\Pic(D_M)$ is trivial since $D_M$ is local. We claim that $\Pic(\Int(T))=\Pic(T)$.

Let $P$ be a maximal ideal of $T$. Then, $T_P$ is a valuation domain of dimension strictly greater than $1$, and thus by \cite[Proposition I.3.16]{intD} we have $\Int(T_P)=T_P[X]$; hence also $\Int(T)=T[X]$. Since $T$ is integrally closed, the natural map $\Pic(T)\longrightarrow\Pic(T[X])$ is an isomorphism \cite[Corollary 6.1.5]{fontana_libro}, and thus the quotient $\picpol(T)=\frac{\Pic(\Int(T))}{\Pic(T)}$ is trivial. Hence, the sequence in the statement is exact.

To show that it is split, it is enough to note that $\Pic(\Int(D_M))$ is always a free group (if $D_M$ is not discrete since in that case $\Pic(\Int(D_M))=\Pic(D_M[X])=(0)$, if $D_M$ is discrete by \cite[Proposition 7.7]{chabert-pic-intV}). The isomorphism follows.
\end{proof}

\section{The surjectivity of the extension map}
A consequence of Theorem \ref{teor:ext-pic-jaffard} (or rather, of its proof) is that when $T$ is a Jaffard overring then the extension map $\Pic(\Int(D))\longrightarrow\Pic(\Int(D)T)=\Pic(\Int(T))$ is surjective. This property is in general not true, not even for an extension map $\Pic(D)\longrightarrow\Pic(T)$ where $D\subseteq T$ is a flat extension: $D$ may be a local ring (so $\Pic(D)$ is trivial), while the Picard group of a flat overring may not be trivial. Moreover, even if the surjectivity hold, it need not to pass to integer-valued polynomials: we will give in Example \ref{ex:weakJaff} below an example where $\Pic(D)\longrightarrow\Pic(T)$ is surjective, while $\Pic(\Int(D))\longrightarrow\Pic(\Int(T))$ is not. In this section, we collect some sufficient conditions for this surjectivity to hold, which will be useful later, as well as a direct application to the calculation of $\Pic(\Int(D))$ for one-dimensional domains.

\begin{lemma}\label{lemma:sublattice-surj-Pic}
Let $D$ be an integral domain, $T$ a flat overring of $D$, and let $\mathcal{L}$ be a sublattice of $\Over(D)$ such that $\bigcup\{S\mid S\in\mathcal{L}\}=T$. If the extension map $\Pic(D)\longrightarrow\Pic(S)$ is surjective for every $S\in\mathcal{L}$, then the extension map $\Pic(D)\longrightarrow\Pic(T)$ is surjective.
\end{lemma}
\begin{proof}
We first note that, for every finite subset $A\subseteq T$, there is an $S\in\mathcal{L}$ containing $A$: indeed, each $a\in A$ is contained in some $S_a\in\mathcal{L}$, and since $\mathcal{L}$ is a sublattice of $\Over(D)$, there is an $S\in\mathcal{L}$ containing all $S_a$ and thus all of $A$.

Let $I:=(x_1,\ldots,x_n)$ be an invertible ideal of $T$, and let $J:=(y_1,\ldots,y_m)$ be its inverse. Then, $x_iy_j\in T$ for every $i,j$, and there are $r_{ij}\in T$ such that $1=\sum_{i,j}r_{ij}x_iy_j$. Therefore, there is an $S\in\mathcal{L}$ that contains all $x_i$, all $x_iy_j$ and all $r_{ij}$.

Consider $I_0:=(x_1,\ldots,x_n)S$ and $J_0:=(y_1,\ldots,y_m)S$; then, by construction, $I_0J_0\subseteq S$ and $1\in I_0J_0$. Hence, $I_0J_0=S$, so $I_0$ is invertible in $S$. Clearly $I_0T=I$. By hypothesis, there is an invertible ideal $I_1$ of $D$ such that $[I_1S]=[I_0]$; thus, $[I_1T]=[I_1ST]=[I_0T]=[I]$. It follows that the extension map $\Pic(D)\longrightarrow\Pic(T)$ is surjective, as claimed.
\end{proof}

\begin{lemma}\label{lemma:sublattice-surj-PicInt}
Let $D$ be an integral domain, $T$ a flat overring of $D$, and let $\mathcal{L}$ be a sublattice of $\Over(D)$ such that $\bigcup\{S\mid S\in\Lambda\}=T$. If the extension map $\Pic(\Int(D))\longrightarrow\Pic(\Int(D)S)$ is surjective for every $S\in\mathcal{L}$, then the extension map $\Pic(\Int(D))\longrightarrow\Pic(\Int(D)T)$ is surjective.
\end{lemma}
\begin{proof}
Let $\mathcal{L}_1:=\{\Int(D)S\mid S\in\Lambda\}$: then, $\mathcal{L}_1$ is a sublattice of $\Over(\Int(D))$. We claim that its union is $\Int(D)T$. Indeed, if $h\in\Int(D)T$ then $h=f_1t_1+\cdots+f_nt_n$ for some $f_i\in\Int(D)$, $t_i\in T$; if $S\in\mathcal{L}$ contains $t_1,\ldots,t_n$, then $h\in\Int(D)S$. Hence, we can apply Lemma \ref{lemma:sublattice-surj-Pic} to $\mathcal{L}_1$.
\end{proof}

We shall apply this criteria in Propositions \ref{prop:pic-weakjaff-surj} and \ref{prop:picTalpha-surj} below; we conclude this section by showing that for one-dimensional Pr\"ufer domains we can exclude some maximal ideals with infinite residue field while controlling the change in the Picard group.

\begin{lemma}\label{lemma:Pruf1dim-surj-PicInt}
Let $D$ be a one-dimensional Pr\"ufer domain, and let $T$ be a flat overring of $D$. Then, the extension map $\Pic(\Int(D))\longrightarrow\Pic(\Int(D)T)$ is surjective.
\end{lemma}
\begin{proof}
Let $\mathcal{L}$ be the family of all Jaffard overrings of $D$ contained in $T$. Then, $\mathcal{L}$ is a sublattice of $\Over(D)$, since the product of two Jaffard overrings is a Jaffard overring, and the extension map $\Pic(\Int(D))\longrightarrow\Pic(\Int(D)S)=\Pic(\Int(S))$ is surjective for all such $S$.

Take $a\in T$, and let $I:=(D:_Da)=a^{-1}D\cap D$. Since $D$ is a Pr\"ufer domain, $I$ is finitely generated; therefore, both the closed set $V(I)$ and the open set $D(I)\cap\Max(D)$ of $\Max(D)$ are compact in the Zariski topology of $\Max(D)$. Let $\Theta_1:=\{P\in\Max(D)\mid P\in V(I)\}$ and $\Theta_2:=\{Q\in\Max(D)\mid Q\in D(I)\cap\Max(D)\}$, and let $S_i:=\bigcap\{D_P\mid P\in\Theta_1\}$. Applying \cite[Propostion 4.8]{jaff-derived} to $\Theta:=\{D_M\mid M\in\Max(D)\}$, we obtain that $\{S_1,S_2\}$ is a pre-Jaffard family of $D$; being finite, it is a Jaffard family, and thus $S_1$ and $S_2$ are Jaffard overrings.

By construction, $a\in D_Q$ for every $Q\in\Theta_2$, and thus $a\in S_2$. Moreover, if $P$ is a maximal ideal of $D$ such that $PT\neq T$, then $a\in D_P$, and thus $P\in\Theta_2$; hence $S_2\subseteq T$. It follows that $S_2\in\mathcal{L}$, and thus $a$ belongs to the union of the elements of $\mathcal{L}$. Since $a$ was arbitrary, $T$ is equal to the union, and we can apply Lemma \ref{lemma:sublattice-surj-PicInt}.
\end{proof}
\begin{prop}
Let $D$ be a one-dimensional Pr\"ufer domain, let $X:=\{M\in\Max(D)\mid \Int(D_M)\neq D_M[X]\}$ and let $T:=\bigcap\{D_M\mid M\in X\}$. Then, there is an exact sequence
\begin{equation*}
0\longrightarrow\Pic(D,T)\longrightarrow\Pic(\Int(D))\longrightarrow\Pic(\Int(D)T)\longrightarrow 0.
\end{equation*}
In particular, $\picpol(D)\simeq\picpol(D,T)$.
\end{prop}
\begin{proof}
Let $\overline{X}$ be the closure of $X$ in $\Max(D)$, with respect to the inverse topology. Then, $\overline{X}$ is a closed set of $\Spec(D)$, with respect to the inverse topology, and thus it is compact in the Zariski topology; hence, also $\{D_M\mid M\in\overline{X}\}$ is compact, since it is homeomorphic to $\overline{X}$. Moreover, $T=\bigcap\{D_P\mid P\in\overline{X}\}$. By \cite[Proposition 4.8]{jaff-derived} $\Theta:=\{T\}\cup\{D_N\mid N\in\Max(D)\setminus\overline{X}\}$ is a pre-Jaffard family of $D$. By Proposition \ref{prop:prejaff-exact}, there is an exact sequence
\begin{equation*}
0\longrightarrow\Pic(D,\Theta)\longrightarrow\Pic(\Int(D))\longrightarrow\Pic(\Int(D)T)\oplus\prod_{N\in\Max(D)\setminus\overline{X}}\Pic(\Int(D)D_N).
\end{equation*}

Let $N\in\Max(D)\setminus\overline{X}$. By definition, $\Int(D_N)=D_N[X]$, and thus
\begin{equation*}
D_N[X]\subseteq\Int(D)D_N\subseteq\Int(D_N)D_N=D_N[X]D_N=D_N[X];
\end{equation*}
hence $\Pic(\Int(D)D_N)=\Pic(D_N[X])=\Pic(D_N)=(0)$ since $D_N$ is local and integrally closed. Hence, the direct product in the previous sequence vanishes. Moreover, $ID_N$ is principal for every invertible ideal $I$ of $D$; hence, $\Pic(D,\Theta\setminus\overline{X})=\Pic(D)$, and $\Pic(D,\Theta)=\Pic(D,T)$. Thus, the exact sequence becomes
\begin{equation*}
0\longrightarrow\Pic(D,T)\longrightarrow\Pic(\Int(D))\longrightarrow\Pic(\Int(D)T).
\end{equation*}
To conclude, we note that the rightmost map of the sequence is the extension map, which is surjective by Lemma \ref{lemma:Pruf1dim-surj-PicInt}. Hence, the sequence of the statement is exact.

To prove the isomorphism, we apply the same method of Theorem \ref{teor:ext-jaff-picpol}: there is a commutative diagram
\begin{equation*}
\begin{tikzcd}
0\arrow[r] & \Pic(D,T)\arrow[d,equal]\arrow[r] & \Pic(D) \arrow[r]\arrow[d] & \Pic(T)\arrow[d]\arrow[r] & 0\\
0\arrow[r] & \Pic(D,T)\arrow[r] & \Pic(\Int(D))\arrow[r] & \Pic(\Int(D)T)\arrow[r] & 0\\
\end{tikzcd}
\end{equation*}
The rows are exact (by definition and by the first part of the proof), while the vertical maps are injective (and the leftmost one is the identity). By the snake lemma, the cokernels of the other two vertical maps are isomorphic; since they are, respectively, $\picpol(D)$ and $\picpol(D,T)$, the claim is proved.
\end{proof}

\section{Weak Jaffard families}\label{sect:weakJaff}
We now start to study how to extend Theorem \ref{teor:ext-pic-jaffard} towards weak Jaffard and pre-Jaffard families. In these cases, we have two problems: first, the equality $\Int(D)T=\Int(T)$ may not hold (see Example \ref{ex:weakJaff} below); second, the cokernel of the map $\Pic(D)\longrightarrow\Pic(\Int(D))$ cannot reduce to the direct sum, and in general it may be difficult to actually determine it inside the direct product of the various $\Pic(\Int(T))$ or $\Pic(\Int(D)T)$. The first problem cannot be resolved with our methods, and, for the most part, we will have to use the equality $\Int(D)T=\Int(T)$ as an additional hypothesis; to solve the second problem, on the other hand, our strategy will be to write the cokernel as the middle element of some \emph{other} exact sequences, using this knowledge to write exact sequences involving the int-polynomial Picard groups.

We study in this section the case of weak Jaffard families, which will then be used as an inductive step in the next section (where we will deal with pre-Jaffard families).

\begin{prop}\label{prop:pic-weakjaff-surj}
Let $D$ be an integral domain and let $\Theta$ be a weak Jaffard family of $D$ pointed at $T$. Then, the extension maps $\Pic(D)\longrightarrow\Pic(T)$ and $\Pic(\Int(D))\longrightarrow\Pic(\Int(D)T)$ are surjective.
\end{prop}
\begin{proof}
Let $\mathcal{L}$ be the lattice of Jaffard overrings of $D$ contained in $T$. Then, the extension maps $\Pic(D)\longrightarrow\Pic(S)$ and $\Pic(\Int(D))\longrightarrow\Pic(\Int(D)S)=\Pic(\Int(S))$ are surjective for every $S\in\mathcal{L}$. We claim that $\bigcup\{S\mid S\in\mathcal{L}\}=T$.

Indeed, let $a\in T$, and consider $(D:_Da)$. Then, $(D:_Da)T=T$, and thus by \cite[Proposition 5.3(a)]{jaff-derived} there are only finitely many $R\in\Theta$ such that $(D:_Da)R\neq R$, say $R_1,\ldots,R_n$. Define $A:=\bigcap\{R\in\Theta, R\neq R_1,\ldots,R_n\}$: then, $\{A,R_1,\ldots,R_n\}$ is a complete and independent finite family of flat overrings of $D$, and thus it is a Jaffard family; moreover, $(D:_Da)A=A$. In particular, $A$ is a Jaffard overring of $D$ contained in $T$ (hence, $A\in\mathcal{L}$) such that $a\in A$: it follows that $\bigcup\{S\mid S\in\mathcal{L}\}=T$.

The claims now follow from Lemmas \ref{lemma:sublattice-surj-Pic} and \ref{lemma:sublattice-surj-PicInt}.
\end{proof}

\begin{prop}\label{prop:weakJaff}
Let $\Theta$ be a weak Jaffard family of $D$ pointed at $T_\infty$. Let $\pi_\Theta:\Pic(\Int(D))\longrightarrow\prod\{\Pic(\Int(D)T)\mid T\in\Theta\}$ be the extension map and let $\Delta$ be its cokernel. Then, there is an exact sequence
\begin{equation*}
0\longrightarrow\bigoplus_{T\in\Theta\setminus\{T_\infty\}}\Pic(\Int(T))\longrightarrow\Delta\longrightarrow\Pic(\Int(D)T_\infty)\longrightarrow 0.
\end{equation*}
\end{prop}
\begin{proof}
We first note, that, for each $T\in\Theta\setminus\{T_\infty\}$, we have $\Int(D)T=\Pic(\Int(T))$ by Proposition \ref{prop:localizz}\ref{prop:localizz:jaff}, and thus $\Pic(\Int(D)T)=\Pic(\Int(T))$ for these overrings.

The inclusion $\Delta\subseteq\prod\{\Pic(\Int(D)T)\mid T\in\Theta\}$ induces a projection map $\pi':\Delta\longrightarrow\Int(D)T_\infty$, whose kernel contains exactly the extensions of the classes $[I]\in\Pic(\Int(D))$ such that $I$ becomes principal in each $T\in\Theta\setminus\Lambda$. We claim that this kernel is exactly $\bigoplus\{\Pic(\Int(T))\mid T\in\Lambda\}$.

We first show that the direct sum belongs to the kernel, and to do so we need to show that it is actually inside $\Delta$. If $[J]\in\Pic(\Int(T))$ for some $T\in\Theta$, $T\neq T_\infty$, we can consider the Jaffard family $\{T,T^\perp\}$, where $T^\perp:=\bigcap\{S\mid S\in\Theta\setminus\{T\}\}$: then, by Theorem \ref{teor:ext-pic-jaffard}, there is a class $[I]\in\Int(\Pic(D))$ such that $[IT]=[J]$ and $[IT^\perp]=[T^\perp]$, so that $[IS]=[S]$ for all other $S\in\Theta$. Thus the direct sum is contained in the kernel.

Conversely, suppose that $[I]\in\ker\pi'$. Then, $[IT_\infty]=[T_\infty]$, and thus there is an $f\in I$ such that $IT_\infty=fT_\infty$. Hence, $I':=f^{-1}I$ is an integral ideal of $\Int(D)$ such that $I'T_\infty=T_\infty$; by \cite[Proposition 5.3(a)]{jaff-derived}, $I'T\neq T$ for only finitely many $T\in\Theta$. Since $[I']=[I]$, it follows that $\pi_\Theta([I])$  belongs to the direct sum. Therefore, $\ker\pi'=\bigoplus_{T\in\Lambda}\Pic(\Int(T))$, and the claim is proved.

To conclude, we need to show that the map $\Delta\longrightarrow\Int(D)T_\infty$ is surjective. However, this map factorizes the extension map $\Int(D)\longrightarrow\Int(D)T_\infty$, which is surjective by Proposition \ref{prop:pic-weakjaff-surj}, and thus it is surjective itself. The claim is proved.
\end{proof}

We now transform this result using int-polynomial Picard groups; the following lemma has the same role of Proposition \ref{prop:jaffard-PicDTheta}.
\begin{lemma}\label{lemma:weakjaff-pic}
Let $\Theta$ be a weak Jaffard family pointed at $T_\infty$. Then, there is an exact sequence
\begin{equation*}
0\longrightarrow\bigoplus_{T\in\Theta\setminus\{T_\infty\}}\Pic(T)\longrightarrow\frac{\Pic(D)}{\Pic(D,\Theta)}\longrightarrow\Pic(T_\infty)\longrightarrow 0.
\end{equation*}
\end{lemma}
\begin{proof}
The extension map $\Pic(D)\longrightarrow\Pic(T_\infty)$ is surjective, with kernel $\Pic(D,T_\infty)$. In particular, the kernel contains $\Pic(D,\Theta)$, and thus the extension map induces a surjective map $\displaystyle{\frac{\Pic(D)}{\Pic(D,\Theta)}\longrightarrow\Pic(T_\infty)}$ with kernel $\displaystyle{\frac{\Pic(D,T_\infty)}{\Pic(D,\Theta)}}$. We claim that this group is isomorphic to $\displaystyle{\bigoplus_{T\in\Theta\setminus\{T_\infty\}}\Pic(T)}$.

Indeed, consider the extension map $\displaystyle{\phi:\Pic(D,T_\infty)\longrightarrow\bigoplus_{T\in\Theta\setminus\{T_\infty\}}\Pic(T)}$. Note that $\phi$ is well-defined since, if $IT_\infty=T_\infty$, then $IS\neq S$ only for finitely many $S\in\Theta$ \cite[Proposition 5.3(a)]{jaff-derived}. Moreover, $\phi$ is surjective: indeed, let $[I]\in\Pic(T)$, with $I\subseteq T$, and set $J:=I\cap D$. Then, $J$ is an invertbile ideal of $D$ such that $JS=S$ for all $S\in\Theta$, $S\neq T$, and in particular $JT_\infty=T_\infty$. Therefore, $\phi([J])$ is the element of the direct sum whose only nonzero coefficient is the one corresponding to $T$, which is equal to $[I]$. Thus, $\phi$ is surjective.

The kernel of $\phi$ is given by all $[I]\in\Pic(D,T_\infty)$ that become principal in $\Pic(T)$ for each $T\in\Theta$; that is, by definition, $\ker\phi=\Pic(D,\Theta)$. Thus $\displaystyle{\frac{\Pic(D,T_\infty)}{\Pic(D,\Theta)}\simeq\bigoplus_{T\in\Theta\setminus\{T_\infty\}}\Pic(T)}$. The exactness of the sequence of the statement follows.
\end{proof}

\begin{teor}\label{teor:weakjaff-picpol}
Let $\Theta$ be a weak Jaffard family of $D$ pointed at $T_\infty$. Then, there is an exact sequence
\begin{equation*}
0\longrightarrow\bigoplus_{T\in\Theta\setminus\{T_\infty\}}\picpol(T)\longrightarrow\picpol(D)\longrightarrow\picpol(D,T_\infty)\longrightarrow 0.
\end{equation*}
\end{teor}
\begin{proof}
Consider the commutative diagram
\begin{equation*}
\begin{tikzcd}
0\arrow[r] & \displaystyle{\bigoplus_{T\in\Theta\setminus\{T_\infty\}}}\Pic(T) \arrow[r]\arrow[d] & \displaystyle{\frac{\Pic(D)}{\Pic(D,\Theta)}} \arrow[r]\arrow[d] & \Pic(T_\infty)\arrow[d] \arrow[r] & 0\\
0\arrow[r] & \displaystyle{\bigoplus_{T\in\Theta\setminus\{T_\infty\}}}\Pic(\Int(T)) \arrow[r] & \Delta\arrow[r] & \Pic(\Int(D)T_\infty)\arrow[r] & 0\\
\end{tikzcd}
\end{equation*}
where $\Delta$ is the cokernel of $\pi_\Theta$. The first row is defined (and is exact) by Lemma \ref{lemma:weakjaff-pic}, while the second row is exact by Proposition \ref{prop:weakJaff}. All vertical maps are injective: the side ones since $\Pic(A)\longrightarrow\Pic(\Int(D)A)$ is always injective, while the middle one because the kernel of the natural map $\Pic(D)\longrightarrow\Delta$ is exactly $\Pic(D,\Theta)$. By the snake lemma, the sequence of cokernels
\begin{equation*}
0\longrightarrow\bigoplus_{T\in\Theta\setminus\{T_\infty\}}\picpol(T)\longrightarrow G\longrightarrow \picpol(D,T_\infty)\longrightarrow 0
\end{equation*}
is exact. Moreover, the quotient $G$ between $\Delta$ and $\displaystyle{\frac{\Pic(D)}{\Pic(D,\Theta)}}$ is isomorphic to
\begin{equation*}
\frac{\Pic(\Int(D))/\Pic(D,\Theta)}{\Pic(D)/\Pic(D,\Theta)}\simeq\frac{\Pic(\Int(D))}{\Pic(D)}=\picpol(D);
\end{equation*}
thus, we obtain exactly the exact sequence of the statement.
\end{proof}

\begin{ex}\label{ex:weakJaff}
Let $p$ be a prime number, and let $V:=\insZ_{(p)}$. Applying repeatedly \cite[Chapter 6, Theorem 4]{ribenboim}, we can construct a chain of extensions $\insQ=L_0\subset L_1\subset\cdots\subset L_n\subset\cdots$ such that, for every $n$, $V$ has $n+1$ extensions $W_n,Z_{1,n},\ldots,Z_{n,n}$ to $L_n$:
\begin{itemize}
\item $W_n$ extends to $W_{n+1}$ and $Z_{n+1,n+1}$ in $L_{n+1}$;
\item for $i=1,\ldots,n$, $Z_{i,n}$ has a unique extension to $L_{n+1}$, namely $Z_{i,n+1}$;
\item $V\subset W_n$ is an immediate extension;
\item for each $i$, the extension $V\subset Z_{i,n}$ is trivial on value groups, while the extension of residue fields has degree at least $n$.
\end{itemize}
Let $L:=\bigcup_nL_n$. Then, the integral closure $\overline{V}$ of $V$ in $L$ is a one-dimensional Pr\"ufer domain whose localization at the maximal ideals are the extensions of $V$ to $L$, namely $W_\infty:=\bigcup_n W_n$ and, for each $i\inN$, $Z_{i,\infty}:=\bigcup_nZ_{i,n}$.

Each $Z_{i,\infty}$ is an isolated point of $\Zar(\overline{V})$ (because there is a $z\in Z_{i,i}\setminus(W_i\cup Z_{1,i}\cup\cdots\cup Z_{i-1,i})$), while $W_\infty$ is not isolated, since every finite subset of $W_\infty$ is contained in some $W_k$ and thus also in $W_{k,\infty}$. In particular, $\overline{V}$ is equal to the intersection of all $W_{i,\infty}$, and $\Theta:=\{W_\infty,Z_{i,\infty}\mid i\inN\}$ is a weak Jaffard family of $\overline{V}$.

The residue field of each $Z_{i,\infty}$ is infinite, and thus $\Int(Z_{i,\infty})$ is trivial; therefore, also $\Int(\overline{V})$ is trivial, and thus $\Int(\overline{V})W_\infty=W_\infty[X]$. However, $W_\infty$ is a DVR with finite residue field, and thus $\Int(W_\infty)$ is not trivial \cite[Proposition I.3.16]{intD}; it follows that $\Int(\overline{V})W_\infty\neq\Int(W_\infty)$.
\end{ex}

\section{Pre-Jaffard families}\label{sect:preJaff}
Proposition \ref{prop:weakJaff} is, in some ways, the best result that is possible to obtain without adding more hypothesis. However, if $\Int(D)T_\infty=\Int(T_\infty)$ (something that need not to happen, see Example \ref{ex:weakJaff}), then one may in principle repeat the process by taking a weak Jaffard family $\Theta'$ of $T_\infty$ and apply the same result; hopefully, this can lead to a finer description of $\Delta$ and thus of $\Pic(\Int(D))$ and $\picpol(D)$. The purpose of this section is to systematize this idea by using the notions of pre-Jaffard family and of its derived sequence (see Section \ref{sect:prelim:jaff}); we use throughout the section the notation introduced therein.

\begin{lemma}\label{lemma:Talpha-limit-union}
Let $\Theta$ be a pre-Jaffard family of $D$, and let $\gamma$ be a limit ordinal. Then, $\displaystyle{\bigcup_{\gamma<\alpha}T_\gamma=T_\alpha}$.
\end{lemma}
\begin{proof}
Let $R$ be the union of $T_\gamma$, for $\gamma<\alpha$. Then, $R$ is the union of a chain of flat overrings, and thus it is itself flat; moreover, $R\subseteq T_\alpha$ since $T_\gamma\subseteq T_\alpha$ when $\gamma<\alpha$. If $R\neq T_\alpha$, then (since $T_\alpha$ is flat too) there should be a nonzero prime ideal $P$ of $D$ such that $PR\neq R$ and $PT_\alpha=T_\alpha$. Since $\Theta$ is a Jaffard family, there is a unique $T\in\Theta$ such that $PT\neq T$; by construction, $T\notin\njaff^\alpha(\Theta)$, and since $\alpha$ is a limit ordinal there is a $\beta<\alpha$ such that $T\notin\njaff^\beta(\Theta)$. In this case, we have $PT_\beta=T_\beta$, and thus $PR=R$ since $T_\beta\subseteq R$. This is a contradiction, and thus $R=T_\alpha$, as claimed.
\end{proof}

\begin{prop}\label{prop:picTalpha-surj}
Let $\Theta$ be a pre-Jaffard family of $D$, and let $\{T_\alpha\}$ be the derived series of $D$. Then:
\begin{enumerate}[(a)]
\item for each $\alpha$, the extension map $\Pic(D)\longrightarrow\Pic(T_\alpha)$ is surjective;
\item if $\Int(D)T_\gamma=\Int(T_\gamma)$ for every $\gamma<\alpha$, then the extension map $\Pic(\Int(D))\longrightarrow\Pic(\Int(D)T_\alpha)$ is surjective.
\end{enumerate}
\end{prop}
\begin{proof}
We proceed by induction on $\alpha$, considering both cases at the same time. If $\alpha=0$ then $T_\alpha=D$ and the claim is trivial. If $\alpha$ is a limit ordinal, then by Lemma \ref{lemma:Talpha-limit-union} $\displaystyle{\bigcup_{\gamma<\alpha}T_\gamma=T_\alpha}$, and the claim follows by applying the inductive hypothesis and Lemmas \ref{lemma:sublattice-surj-Pic} and \ref{lemma:sublattice-surj-PicInt}  to $\{T_\gamma\mid\gamma<\alpha\}$.

Suppose that $\alpha=\gamma+1$ is a successor ordinal. Then, the extension map $\Pic(D)\longrightarrow\Pic(T_\alpha)$ factors as
\begin{equation*}
\Pic(D)\longrightarrow\Pic(T_\gamma)\longrightarrow\Pic(T_\alpha).
\end{equation*}
The first of these maps is surjective by hypothesis; on the other hand, $\njaff^\gamma(\Theta)$ is a pre-Jaffard family of $T_\gamma$, and thus $T_\alpha$ belongs to the weak Jaffard family $(\njaff^\gamma(\Theta)\setminus\njaff^\alpha(\Theta))\cup\{T_\alpha\}$, which implies that $\Pic(T_\gamma)\longrightarrow\Pic(T_\alpha)$ is surjective by Proposition \ref{prop:pic-weakjaff-surj}. In the same way, $\Pic(\Int(D))\longrightarrow\Pic(\Int(D)T_\alpha)$ factors as 
\begin{equation*}
\Pic(D)\longrightarrow\Pic(\Int(D)T_\gamma)\longrightarrow\Pic(\Int(D)T_\alpha).
\end{equation*}
The first map is surjective by hypothesis; the second one is surjective since $\Pic(\Int(D)T_\gamma)=\Pic(\Int(T_\gamma))$ and thus we can apply again Proposition \ref{prop:pic-weakjaff-surj}. The claim is proved.
\end{proof}

\begin{teor}\label{teor:preJaff}
Let $\Theta$ be a pre-Jaffard family of $D$, and let $\{T_\alpha\}$ be the derived series of $D$. Fix an ordinal $\alpha$ and suppose that the following conditions hold:
\begin{itemize}
\item $\Int(D)T=\Int(T)$ for each $T\in\Theta\setminus\njaff^\alpha(\Theta)$ and for each $T=T_\gamma$ with $\gamma<\alpha$;
\item $\picpol(\Int(T))$ is a free group for each $T\in\Theta\setminus\njaff^\alpha(\Theta)$.
\end{itemize}
Then, there is an exact sequence
\begin{equation*}
0\longrightarrow\bigoplus_{T\in\Theta\setminus\njaff^\alpha(\Theta)}\picpol(T)\longrightarrow\picpol(D)\longrightarrow\picpol(D,T_\alpha)\longrightarrow 0.
\end{equation*}
\end{teor}
\begin{proof}
By induction on $\alpha$. If $\alpha=1$ then $\Lambda_\alpha=(\Theta\setminus\njaff^\alpha(\Theta))\cup\{T_\alpha\}$ is a weak Jaffard family of $D$, and thus the statement is exactly Theorem \ref{teor:weakjaff-picpol}.

Suppose that $\alpha=\gamma+1$ is a successor ordinal. There is a commutative diagram
\begin{equation}\label{eq:diagpreJaff}
\begin{tikzcd}
0\arrow[r] & \displaystyle{\bigoplus_{T\in\Theta\setminus\njaff^\gamma(\Theta)}\picpol(T)}\arrow[r]\arrow[d,"f"] &\picpol(D)\arrow[r]\arrow[d,equal] &\picpol(T_\gamma)\arrow[r]\arrow[d,"g"] & 0\\
0\arrow[r] & L\arrow[r] &\picpol(D)\arrow[r] &\picpol(D,T_\alpha)\arrow[r] & 0,
\end{tikzcd}
\end{equation}
where $L$ is the kernel of $\picpol(D)\longrightarrow\picpol(D,T_\alpha)$; note that this map is surjective since $\Pic(\Int(D))\longrightarrow\Pic(\Int(D)T_\alpha)$ is surjective by Proposition \ref{prop:picTalpha-surj}.

The first row is exact by induction (using the hypothesis $\Int(D)T_\gamma=\Int(T_\gamma)$ and thus $\picpol(D,T_\gamma)=\picpol(T_\gamma)$), while the second one is exact by definition of $L$. Since the map in the middle column is an equality, its kernel and cokernel are trivial, and thus by the snake lemma $\coker f\simeq\ker g$; by Theorem \ref{teor:weakjaff-picpol}, the latter is isomorphic to $\bigoplus\{\picpol(T)\mid T\in\njaff^\gamma(\Theta)\setminus\njaff^\alpha(\Theta)\}$, and thus there is an exact sequence
\begin{equation*}
0\longrightarrow \bigoplus_{T\in\Theta\setminus\njaff^\gamma(\Theta)}\picpol(T)\longrightarrow L\longrightarrow \bigoplus_{T\in\njaff^\gamma(\Theta)\setminus\njaff^\alpha(\Theta)}\picpol(T)\longrightarrow 0.
\end{equation*}
By hypothesis, each $\picpol(T)$ is free; hence the sequence splits and thus $L$ is isomorphic to the direct sum of $\picpol(T)$ for $T\in\Theta\setminus\njaff^\alpha(\Theta)$. The claim now follows reading the second row of \eqref{eq:diagpreJaff}.

\medskip

Suppose now that $\alpha$ is a limit ordinal; for each $\gamma\leq\alpha$, let $L_\gamma$ be the kernel of the surjective map $\picpol(D)\longrightarrow\picpol(D,T_\gamma)$. By induction, $\{L_\gamma\}_{\gamma<\alpha}$ is a chain of free subgroups of $L_\alpha$ such that each element is a direct summand of the next ones; we claim that $\bigcup_{\gamma<\alpha}L_\gamma=L_\alpha$.

Let $k\in L_\alpha$: then, $k$ is the image in $\picpol(D,T_\alpha)$ of an invertible ideal $I:=(f_1,\ldots,f_n)$ of $\Int(D)$ such that $I\Int(D)T_\alpha$ is principal, say generated by $g$. In particular, there are $t_1,\ldots,t_n\in\Int(D)T_\alpha$ such that $g=f_1t_1+\cdots+f_nt_n$, and $f_ig^{-1}\in\Int(D)T_\alpha$ for every $i$. By Lemma \ref{lemma:Talpha-limit-union}, $T_\alpha$ is the union of $T_\gamma$, for $\gamma<\alpha$, and thus the same holds for $\Int(D)T_\alpha$ and $\Int(D)T_\gamma$; therefore, there is a $\overline{\gamma}<\alpha$ such that $\Int(D)T_{\overline{\gamma}}$ contains all $t_i$ and all $f_ig^{-1}$. Then, $I\Int(D)T_{\overline{\gamma}}$ is a principal ideal, generated by $g$; in particular, the image of $k$ in $\picpol(D,T_{\overline{\gamma}})$ is trivial, i.e., $g\in L_{\overline{\gamma}}$.

Therefore, we can apply \cite[Lemma 5.6]{almded-radfact} (or \cite[Chapter 3, Lemma 7.3]{fuchs-abeliangroups}), obtaining that $L_\alpha\simeq\bigoplus\{\picpol(T)\mid T\in\Theta\setminus\njaff^\alpha(\Theta)\}$. The claim is proved.
\end{proof}

\begin{cor}\label{cor:Thetasharp}
Let $\Theta$ be a pre-Jaffard family of $D$ such that:
\begin{itemize}
\item $\Int(D)T=\Int(T)$ for every $T\in\Theta$;
\item $\picpol(\Int(T))$ is free for every $T\in\Theta$;
\item $\Theta$ is sharp.
\end{itemize}
Then, $\displaystyle{\picpol(D)\simeq\bigoplus_{T\in\Theta}\picpol(T)}$.
\end{cor}
\begin{proof}
The first condition implies, thanks to Proposition \ref{prop:localizz-intersec}, that $\Int(D)T_\alpha=T_\alpha$ for every $\alpha$; moreover, together with the second condition, it also implies that we can apply Theorem \ref{teor:preJaff}. Since $\Theta$ is sharp, there is an $\alpha$ such that $\njaff^\alpha(\Theta)=\emptyset$, i.e., $T_\alpha=K$; for this $\alpha$, $\Int(D)T_\alpha=\Int(D)K=K[X]$, and thus $\picpol(D,T_\alpha)=(0)$. The claim follows from Theorem \ref{teor:preJaff}.
\end{proof}

The condition that $\picpol(\Int(T))$ is free is satisfied, for example, when $T=D_M$ is a discrete valuation ring. A ring such that all the localizations at the maximal ideals are DVRs is called an \emph{almost Dedekind domain}; the following two results apply Theorem \ref{teor:preJaff} to this class of rings. We note that it is possible to characterize for which almost Dedekind domains the ring of integer-valued polynomials behave well under localization \cite[Theorem 4.3]{chabert-localizationInt}.

\begin{teor}\label{teor:almded}
Let $D$ be an almost Dedekind domain, $\{T_\alpha\}$ be the derived series of the canonical pre-Jaffard family $\Theta:=\{D_M\mid M\in\Max(D)\}$ of $D$, corresponding to $\njaff^\alpha(\Theta)\subseteq\Max(D)$. If, for every $M\notin\njaff^\alpha(\Theta)$, we have $\Int(D)D_M=\Int(D_M)$, and $\Int(D)T_\alpha=\Int(T_\alpha)$, then there is an exact sequence
\begin{equation*}
0\longrightarrow\bigoplus_{M\notin\njaff^\alpha(\Theta)}\Pic(\Int(D_M))\longrightarrow\picpol(D)\longrightarrow\picpol(T_\alpha)\longrightarrow 0.
\end{equation*}
\end{teor}
\begin{proof}
The condition on localization implies, by Proposition \ref{prop:localizz-intersec}, that $\Int(D)T_\gamma=T_\gamma$ for every $\gamma<\alpha$. The claim now follows from Theorem \ref{teor:preJaff}.
\end{proof}

\begin{cor}\label{cor:almded-best}
Let $D$ be an almost Dedekind domain. If $\Int(D)D_M=\Int(D_M)$ for all $M\in\Max(D)$ and $\Max(D)$ is scattered (with respect to the inverse topology) then 
\begin{equation*}
\picpol(D)\simeq\bigoplus_{M\in\Max(D)}\picpol(D_M)
\end{equation*}
and
\begin{equation*}
\Pic(\Int(D))\simeq\Pic(D)\oplus\bigoplus_{M\in\Max(D)}\Pic(\Int(D_M))
\end{equation*}
\end{cor}
\begin{proof}
If $\Max(D)$ is scattered, then the canonical pre-Jaffard family $\Theta:=\{D_M\mid M\in\Max(D)\}$ is sharp \cite[Corollary 8.6]{jaff-derived}. The claim now follows from Corollary \ref{cor:Thetasharp} (or by Theorem \ref{teor:almded} applied with $\alpha$ being the Cantor-Bendixson rank of $\Max(D)$, endowed with the inverse topology).
\end{proof}

\end{document}